\documentclass[reqno]{amsart}
\usepackage{enumerate}
\usepackage{amssymb,url,color}
\usepackage{hyperref}
\usepackage{mathrsfs}
\newtheorem{thm}{Theorem}[section]
\newtheorem{cor}{Corollary}[section]

\newtheorem{lem}{Lemma}[section]
\newtheorem{rem}{Remark}[section]
\theoremstyle{Problem}

\theoremstyle{definition}

\numberwithin{equation}{section}

\def\beq{\begin{equation}}
\def\deq{\end{equation}}
\allowdisplaybreaks 

\def\cB{{\mathcal B}}

\def\cM{{\mathcal M}}

\def\mC{{\mathbb C}}

\def\mE{{\mathbb E}}

\def\mI{{\mathbb I}}

\def\mN{{\mathbb N}}

\def\mP{{\mathbb P}}
\def\mQ{{\mathbb Q}}
\def\mR{{\mathbb R}}
\def\mS{{\mathbb S}}

\def\sB{{\mathscr B}}

\def\sF{{\mathscr F}}

\def\geq{\geqslant}
\def\leq{\leqslant}

\def\e{{\mathrm{e}}}

\def\[{{\Big[}}
\def\]{{\Big]}}
\def\<{{\langle}}
\def\>{{\rangle}}
\def\({{\Big(}}
\def\){{\Big)}}

\def\dif{{\rm d}}

\def\={&\!\!=\!\!&}
\def\bt{\begin{theorem}}
\def\et{\end{theorem}}
\def\bl{\begin{lemma}}
\def\el{\end{lemma}}
\def\br{\begin{rem}}
\def\er{\end{rem}}

\bfseries\rmfamily 

\begin{document}

\title
[Degenerate SDEs with Sobolev Diffusion coefficients]
{Strong Uniqueness of Degenerate SDEs with H\"older diffusion coefficients}

\author{Zhen Wang and Xicheng Zhang}

\address{Zhen Wang: School of Mathematics and Statistics, Wuhan University, Wuhan, Hubei, 430072, P.R.China}
\email{wangzhen881025@163.com}

\address{Xicheng Zhang: School of Mathematics and Statistics, Wuhan University, Wuhan, Hubei, 430072, P.R.China}
\email{XichengZhang@gmail.com}

\thanks{This work is partially supported by an NNSFC grant of China (No. 11731009). }

\begin{abstract}
In this paper we prove a new strong uniqueness result and a weak existence result
for possibly {\it degenerate} multidimensional stochastic differential equations with Sobolev diffusion coefficients and rough drifts.
In particular, examples with H\"older diffusion coefficients are provided to show our results.
\end{abstract}

\subjclass[2000]{}

\maketitle

\section{Introduction}

Consider the following stochastic differential equation in $\mR^d$:
\beq\label{1}
\dif X_t=b(t,X_t)\dif t+\sigma(t,X_t)\dif W_t, \quad X_0=x,
\deq
where $b:\mathbb{R}_+\times\mathbb{R}^d\rightarrow \mathbb{R}^d$
and $\sigma:\mathbb{R}_+\times\mathbb{R}^d\rightarrow
\mathbb{R}^d\otimes\mathbb{R}^m$ are two Borel measurable functions, and $W$ is an $m$-dimensional
standard Brownian motion defined on a complete filtered probability space $(\Omega, \mathscr{F}, \mP; (\sF_t)_{t\geq 0})$.

It is a classical fact that when $b$ and $\sigma$ are uniformly Lipschitz continuous in $t$ with respect to the spatial variable $x$, SDE \eqref{1} admits a unique strong solution. However, when $b$ and $\sigma$ are non-Lipschitz continuous, the pathwise uniqueness would be broken as the case of ordinary differential equations. In one dimensional case, the famous Yamada and Watanabe's theorem \cite{W-Y} provides a sufficient condition for pathwise uniqueness. More precisely, suppose that $d=1$ and for some concave functions
$\gamma,\rho:\mR_+\to\mR_+$,
$$
|b(t,x)-b(t,y)|\leq\gamma(|x-y|),\ |\sigma(t,x)-\sigma(t,y)|\leq \rho(|x-y|),
$$
where $\int_{0+}1/\gamma(s)\dif s=\infty$ and $\int_{0+}1/\rho^2(s)\dif s=\infty$, then pathwise uniqueness holds for SDE \eqref{1} with $d=1$. It is noticed that $\sigma(t,x)=|x|^{1/2}$ satisfies the above condition. While in the multidimentional case, if we require $\int_{0+}1/\rho(s)\dif s=\infty$, 
then pathwise uniqueness still holds (see \cite{W-Y} and \cite{Fa-Zh}). 
Moreover, when $d\geq 2$, the above conditions are almost optimal (see \cite{Sw} for a counter-example).


On the other hand, when $\sigma$ is uniformly nondegenerate,  that is, $\sigma\sigma^*$ is strictly positive, there are many results devoted to the 
study of pathwise uniqueness.
We only mention parts of them. In \cite{V}, Veretennikov showed the study of pathwise uniqueness of SDE \eqref{1}
when $\sigma=\mI$ and $b$ is bounded measurable. In \cite{K-R},
Krylov and R\"ockner showed strong well-posedness of SDE \eqref{1} when $\sigma=\mI$ and $b\in L^q_{loc}(\mR_+; L^p(\mR^d))$ for $2/q+d/p<1$.
In \cite{Zh2} and \cite{Zh3}, we extend Krylov and R\"ockner's result to the multiplicative noise case under Sobolev diffusion and the same drift coefficients.
See also \cite{Zh4} for the study of the weak differentiability of the solutions with respect to the initial values.
Moreover, when $b$ and $\sigma$ are only bounded measurable, and $\sigma$ is uniformly nondegenerate,
the weak existence of SDE \eqref{1} was shown in \cite{K} based on the Krylov estimate established by himself.

However, when $b$ is irregular and $\sigma$ is not uniformly nondegenerate, there are few results to discuss the weak existence and pathwise uniqueness.
In a very special case, Kumar \cite{Ku} and Luo \cite{Lu} obtain the pathwise uniqueness of multidimensional SDEs with H\"older diffusion coefficients. 
Essentially, it is the same as Yamada-Watanabe's result.
It is noticed that in \cite{Ch-Ja}, the authors obtain some special results about the pathwise uniqueness and weak existence based on some results from PDEs' theory.
In this paper, we find simple conditions to guarantee the pathwise uniqueness for multidimensional SDEs without assuming the {\it uniform ellipticity} and Lipschitz continuity. More precisely, we shall show the following results. 
\begin{thm}\label{thm1}
Suppose that there exists a nonnegative measurable function $F$ with
\begin{align}\label{FF}
\int_0^t\int_{B_R}F(s,x)^p\det(\sigma\sigma^*)^{-1}(s,x)\dif x\dif s<\infty,\ t, R>0
\end{align}
for some $p\geq d+1$, and such that for Lebesgue-almost all $s,x,y$:
$$
2\<x-y,b(s,x)-b(s,y)\>+\|\sigma(s,x)-\sigma(s,y)\|^2\leq |x-y|^2(F(s,x)+F(s,y)).
$$
Then the local pathwise uniqueness holds for SDE \eqref{1}.
Moreover, if $b$ and $\sigma$ are time-independent, then the above requirement $p\geq d+1$ can be replaced with $p\geq d$.
\end{thm}

We have the following easy corollary.
\begin{cor}\label{cor5}
Suppose that $\sigma:\mathbb{R}^d\rightarrow\mathbb{R}^d\otimes\mathbb{R}^m$ satisfies for some $p\geq d$,
\begin{align}\label{HG1}
\int_{B_R}(\cM_R|\nabla\sigma|(x))^p\det(\sigma\sigma^*)^{-1}(x)\dif x<\infty,\ R>0,
\end{align}
where $\nabla\sigma$ stands for the generalized gradient, and $\cM_R|\nabla\sigma|$ is the local Hardy-Littlewood maximal function defined by
$$
\mathcal{M}_R\phi(x):=\sup_{0<r<R}\frac{1}{|B_r|}\int_{B_r}\phi(x+y)\dif y.
$$
Then  the local pathwise uniqueness holds for SDE
$$
\dif X_t=\sigma(X_t)\dif W_t,\ \ X_0=x.
$$
\end{cor}
\begin{proof}
It is well known that for all $x,y\in B_R$ (for example, see \cite[Lemma 3.5]{Zh33}),
$$
|\sigma(x)-\sigma(y)|\leq C|x-y|(\cM_R|\nabla\sigma|(x)+\cM_R|\nabla\sigma|(y)).
$$
Now we can apply Theorem \ref{thm1} to conclude the result.
\end{proof}
Below we provide several examples to show our result.

\noindent{\bf Example 1} Let $d> 2n$ with $n\in\mN$ and $\alpha\in(0,1]$, $\beta\in[\alpha,1]$. Consider the following diffusion matrix $\sigma(x)$
$$
 \sigma(x)=
\left( \begin{array}{ccccccc}
|x|^{\alpha}&\cdots & 0 & 0 & \cdots & 0\\
\vdots & \ddots & \vdots  & \vdots & \ddots & \vdots \\
0& \cdots & |x|^{\alpha} & 0 & \cdots & 0\\
0 & \cdots & 0 & |x|^\beta+1 & \cdots & 0 \\
\vdots & \ddots & \vdots & \vdots& \vdots & \vdots \\
0 & \cdots & 0 & 0 &\cdots & |x|^\beta+1 \\
\end{array}\right)
:=\left( \begin{array}{cc}
A& 0\\
0& B
\end{array}\right),
$$
where $A$ is a $n\times n$ matrix and $B$ is a $(d-n)\times(d-n)$ matrix.
One can check that \eqref{HG1} holds.
Indeed, it is easy to see that
\begin{align}\label{HG3}
\det (\sigma\sigma^{*})(x)=|x|^{2n\alpha}(|x|^\beta+1)^{2(d-n)}\geq |x|^{2n\alpha}.
\end{align}
Moreover, for $\gamma\in(0,1)$, we have the following easy fact: 
\begin{align}\label{HG8}
\sup_{s>0}\frac{1}{|B_s|}\int_{B_s}|x+y|^{-\gamma}\dif y\leq c|x|^{-\gamma},
\end{align}
which follows by the following observation
$$
\frac{1}{|B_s|}\int_{B_s}|x+y|^{-\gamma}\dif y
\leq\left\{
\begin{aligned}
&\frac{1}{|B_s|}\int_{B_s}||x|-|y||^{-\gamma}\dif y\leq 2^\gamma|x|^{-\gamma},\ s<|x|/2,\\
&\frac{1}{|B_s|}\int_{B_{4s}}|y|^{-\gamma}\dif y\leq c s^{-\gamma}\leq c|x|^{-\gamma},\ s\geq|x|/2,
\end{aligned}
\right.
$$
where $c=c(d,\gamma)>0$.
Hence, by \eqref{HG3}, \eqref{HG8} and $d>2n$,
\begin{align*}
\int_{B_R}(\cM_R|\nabla\sigma|(x))^d\det(\sigma\sigma^*)^{-1}(x)\dif x
\leq C\int_{B_R}|x|^{(\alpha-1)d-2n\alpha}\dif x<\infty.
\end{align*}
Notice that the function $x\mapsto |x|^\alpha$ is only $\alpha$-order H\"older continuous at point $0$.

\noindent{\bf Example 2}
Let $d=3$ and $m_1,m_2,m_3\geq 2$. Let $\alpha\in(0,1)$
and $y_{ij}\in\mR^d, i=1,\cdots,m_j,\ j=1,2,3$. Consider the following $\sigma(x)$
$$
 \sigma(x)=
\left( \begin{array}{ccc}
\sum_{i=1}^{m_1}|x-y_{i1}|^{\alpha} & 0 &0 \\
0 & \sum_{i=1}^{m_2}|x-y_{i2}|^{\alpha} & 0 \\
0 & 0 & \sum_{i=1}^{m_3}|x-y_{i3}|^{\alpha} \\
\end{array}\right).
$$
As above, one can check that \eqref{HG1} holds. Notice that $\sigma$ is H\"older continuous at points $y_{ij}$.

\noindent{\bf Example 3} Let $\alpha\in(0,1)$. Consider the following one-dimensional SDE:
$$
\dif Z_t=\left(|Z_t|^\alpha+|W^{(2)}_t|^\alpha+|W^{(3)}_t|^\alpha\right)\dif W^{(1)}_t,\ Z_0=z,
$$
where $(W^{(1)}, W^{(2)}, W^{(3)})$ is a three dimensional standard Brownian motion. The above SDE can be written as a 
three dimensional SDE with $X_t=(Z_t, W^{(2)}_t, W^{(3)}_t)$ and
$$
 \sigma(x)=
\left( \begin{array}{ccc}
|x_1|^\alpha+|x_2|^\alpha+|x_3|^\alpha & 0 &0\\
0 & 1 &0\\
0 & 0 &1\\
\end{array}\right).
$$
One can verify that \eqref{HG1} holds 
for the above $\sigma$. It should be observed that the H\"older continuity index can be less than $1/2$ compared with Yamada-Watanabe's 
classical result due to the regularization effect of Brownian noises.

About the weak existence, we have the following result.
\begin{thm}\label{Th2}
Suppose that $b$ and $\sigma$ are locally bounded measurable and linear growth, and
\begin{align}\label{HG2}
\sup_{s\in[0,T]}\int_{B_R}|\det(\sigma\sigma^*)(s,x)|^{-2d-3}\dif x<\infty.
\end{align}
Then SDE \eqref{1} admits a weak solution. More precisely, there is a filtered probability space $(\Omega,\sF,\mP; (\sF_t)_{t\geq 0})$ and two $\sF_t$-adapted
processes $(X,W)$ defined on it so that

(i) $W$ is an $\sF_t$-Brownian motion;

(ii) $(X,W)$ satisfies the following integral equation:
$$
X_t=x+\int^t_0\sigma(s,X_s)\dif W_s+\int^t_0\sigma(s,X_s)\dif s.
$$
Moreover, if $b$ and $\sigma$ are time independent, then \eqref{HG2} can be replaced with
\begin{align}\label{hG5}
\int_{B_R}|\det(\sigma\sigma^*)(x)|^{-2d-1}\dif x<\infty.
\end{align}
\end{thm}
\noindent{\bf Example 4}
Suppose that $b$ is bounded measurable and $\sigma$ takes the following form:
$$
 \sigma(x)=
\left( \begin{array}{cccc}
\left||x|-1\right|^\alpha &0\\
0 & |x|^\beta
\end{array}\right),
$$
where $\alpha\in(0,\frac{1}{10})$ and $\beta\in(0,\frac{1}{5})$. One sees that assumption \eqref{hG5} is satisfied.
Notice that $\sigma$ is H\"older continuous on the sphere $\mS^{d-1}=\{x: |x|=1\}$, and
degenerate on a $d-1$-dimensional submanifold $\mS^{d-1}$.

\section{Proofs}

We first recall some tools of proving our main results.
Consider the following $d$-dimensional It\^o's process:
\begin{align}\label{Ito}
\xi_t=\xi_0+\int^t_0b_s\dif s+\int^t_0\sigma_s\dif W_s,
\end{align}
where $\xi_0\in\sF_0$, $b_s:\mR_+\times\Omega\to\mR^d$ is a locally integrable $\mR^d$-valued measurable adapted process
and $\sigma_s: \mR_+\times\Omega\to\mR^d\otimes\mR^m$ is a locally square integral measurable adapted process.
We recall the following Krylov's estimate \cite[Chapter 2, Theorem 2.2]{K}.
\begin{thm}\label{thm2}
Let $\xi_t$ be an It\^o's process with the form \eqref{Ito}. For $R>0$, let $\tau_R$ be the exit time of $\xi_t$ from the ball $B_R$.
For any $T>0$, there exists constants $C_1, C_2$ such that for all $0\leq t_0<t_1\leq T$
and all $f\in L^{d+1}_{loc}(\mR^{d+1})$, $g\in L^d_{loc}(\mR^d)$,
\beq\label{2}
\mathbb{E}\left(\int^{t_1\wedge\tau_R}_{t_0\wedge\tau_R}\left( \det (\sigma_s\sigma_s^*)\right)^{1/(d+1)}f(s,\xi_s)\dif s
\Big|{\sF_{t_0\wedge\tau_R}}\right)\leq C_1\| f\|_{L^{d+1}([t_0,t_1]\times B_R)},
\deq
\beq\label{22}
\mathbb{E}\left(\int^{t_1\wedge\tau_R}_{t_0\wedge\tau_R}\left( \det(\sigma_s\sigma_s^*)\right)^{1/d}g(\xi_s)ds\Big|{\sF_{t_0\wedge\tau_R}}\right)\leq C_2(t_1-t_0)^{\frac{d}{2p}}\|g\|_{L^d(B_R)}.
\deq
\end{thm}

We also need the following stochastic Gronwall's inequality due to Scheutzew \cite{Sc} (see also \cite[Lemma 3.8]{X-Zh}).
\begin{lem}[Stochastic Gronwall's inequality]\label{im}
Let $\xi(t)$ and $\eta(t)$ be two nonnegative c\`adl\`ag $\sF_t$-adapted processes,
$A_t$ a continuous nondecreasing $\sF_t$-adapted process with $A_0=0$, $M_t$ a  local martingale with $M_0=0$. Suppose that
\begin{align}\label{Gron}
\xi(t)\leq\eta(t)+\int^t_0\xi(s)\dif A_s+M_t,\ \forall t\geq 0.
\end{align}
Then for any $0<q<p<1$ and stopping time $\tau$, we have
\begin{align}
\big[\mE(\xi(\tau)^*)^{q}\big]^{1/q}\leq \Big(\tfrac{p}{p-q}\Big)^{1/q}\Big(\mE \e^{pA_\tau/(1-p)}\Big)^{(1-p)/p}\mE\big(\eta(\tau)^*\big),  \label{gron}
\end{align}
where $\xi(t)^*:=\sup_{s\in[0,t]}\xi(s)$.
\end{lem}

The following lemma is taken from \cite[p. 1, Lemma 1.1]{Po}.
\begin{lem}\label{Le1}
Let $\{\xi(t)\}_{t\in[0,T]}$ be a nonnegative measurable ($\sF_t$)-adapted process. Assume that for all $0\leq s\leq t\leq T$,
$$
\mE\left(\int^t_s\xi(r)\dif r\Bigg|\sF_s\right)\leq\rho(s,t),
$$
where $\rho(s,t)$ is a nonrandom interval function satisfying the following conditions:

(i) $\rho(t_1,t_2)\leq\rho(t_3,t_4)$ if $(t_1,t_2)\subset(t_3,t_4)$;

(ii) $\lim_{h\downarrow 0}\sup_{0\leq s<t\leq T, |t-s|\leq h}\rho(s,t)=\kappa,\ \ \kappa\geq0$.\\
Then for any real $\lambda<\kappa^{-1}$ (if $\kappa=0$, then $\kappa^{-1}=+\infty$),
$$
\mE\exp\left\{\lambda\int^T_0\xi(r)\dif r\right\}\leq C=C(\lambda,\rho, T)<+\infty.
$$
\end{lem}

The following localization result will be used to construct a global weak solution (see \cite{S-V}).
\begin{thm}\label{Ext}
Let $\mC$ be the space of all $\mR^d$-valued continuous functions on $\mR_+$, and $\cB_t(\mC)$ the natural $\sigma$-filtration.
Let $(\mP_n)_{n\in\mN}$ be a family of probability measures on $\mC$ and
$(\tau_n)_{n\in\mN}$ a sequence of nondecreasing stopping times. Let $\tau_0\equiv 0$.
Suppose that for each $n\in\mN$, $\mP_{n}$ equals $\mP_{n-1}$ on $\cB_{\tau_{n-1}}(\mC)$,
and for any $T> 0$,
\begin{align}\label{Lim-Stop}
\lim_{n\to\infty} \mP_n (\tau_n\leq T)=0.
\end{align}
Then there is a unique probability measure $\mP$ over $\cB_\infty(\mC)$ such that  $\mP$ equals $\mP_n$ on $\cB_{\tau_n}(\mC)$ and
$\mP_n$ weakly converges to $\mP$ as $n\to\infty$.
\end{thm}

Now we can give
\begin{proof}[Proof of Theorem \ref{thm1}]
Let $X_t$ and $Y_t$ solve SDE $(\ref{1})$ with starting points $x$ and $y$, respectively.
For $R>0$, define stopping time
$$
\tau_R:=\inf\{t>0: |X_t|\vee|Y_t|>R\}.
$$
Set $Z_t:=X_t-Y_t$. By It\^o's formula and the assumption, we have
$$
\aligned
|Z_{t\wedge\tau_R}|^2&=|x-y|^2+2\int_0^{t\wedge\tau_R}\<Z_s, b(s,X_s)-b(s,Y_s)\>\dif s\\
&\quad+2\int_0^{t\wedge\tau_R}\langle Z_s,\sigma(s,X_s)-\sigma(s,Y_s)\rangle\dif W_s\\
&\quad+\int_0^{t\wedge\tau_R}\|\sigma(s,X_s)-\sigma(s,Y_s)\|^2\dif s\\
&\leq |x-y|^2+\int_0^{t\wedge\tau_R}|Z_s|^2(F(s,X_s)+F(s,Y_s))\dif s\\
&\quad+2\int_0^{t\wedge\tau_R}\langle Z_s,\sigma(s,X_s)-\sigma(s,Y_s)\rangle\dif W_s.
\endaligned
$$
By stochastic Gronwall's inequality \eqref{gron}, for any $0<q<p<1$ and stopping time $\tau'$, we have
\begin{align}\label{LL1}
\mE\left(\sup_{t\in[0,T\wedge\tau']}|Z_{t\wedge\tau_R}|^{2q}\right)\leq C|x-y|^2\left(\mE\e^{pA_{T\wedge\tau'}/(1-p)}\right)^{1-1/p},
\end{align}
where
$$
A_t:=\int^{t\wedge\tau_R}_0 (F(s,X_s)+F(s,Y_s))\dif s.
$$
By Krylov's estimate \eqref{2} and the assumption, we have
\begin{align*}
\mE A_t&=\mE\left(\int^{t\wedge\tau_R}_0\det(\sigma\sigma^*)^{1/p}(s,X_s)(F\det(\sigma\sigma^*)^{-1/p})(s,X_s)\dif s\right)\\
&\leq C\left(\int^{t}_0\int_{B_R}|F(s,x)|^p\det(\sigma\sigma^*)^{-1}(s,x)\dif x\dif s\right)^{1/p}<\infty.
\end{align*}
In particular, $t\mapsto A_t$ is a continuous adapted process and if we define $\tau'_N:=\inf\{t>0: A_t>N\}$, then
$$
\mP\left(\lim_{N\to\infty}\tau'_N=\infty\right)=1.
$$
In \eqref{LL1}, we replace $\tau'$ by $\tau'_N$ and let $x=y$, then
$$
\mE\left(\sup_{t\in[0,T\wedge\tau'_N]}|Z_{t\wedge\tau_R}|^{2q}\right)=0.
$$
Letting $N,R\to\infty$, by Fatou's lemma we get the pathwise uniqueness.
If $b$ and $\sigma$ are time-independent, we can use \eqref{22} instead of \eqref{2} to derive the same result.
\end{proof}

\begin{proof}[Proof of Theorem \ref{Th2}]
We shall use Theorem \ref{Ext} and Girsanov's theorem to construct a solution.
First of all, since $\sigma$ is continuous and linear growth, it is well known that there is a solution $(Y_t, W_t)$ solving the following SDE:
$$
\dif Y_t=\sigma(Y_t)\dif W_t,\  \ Y_0=y.
$$
For $R>0$, define a stopping time
$$
\tau_R:=\{t>0: |Y_t|>R\}
$$
and let $b_\sigma:=\sigma^*(\sigma\sigma^*)^{-1}b$ and
$$
\dif\mQ_R:=\exp\left\{-\int^{T\wedge\tau_R}_0b_\sigma(s,Y_s)\dif W_s-\frac{1}{2}\int^{T\wedge\tau_R}_0|b_\sigma(s,Y_s)|^2\dif s\right\}\dif \mP.
$$
In order to show that the above $\mQ_R$ is a probability measure, by Novikov's criterion it suffices to verify that
\begin{align}\label{LL2}
\mE\exp\left\{\frac{1}{2}\int^{T\wedge\tau_R}_0|b_\sigma(s,Y_s)|^2\dif s\right\}<\infty.
\end{align}
By Krylov's estimate \eqref{2}, there is a constant $C>0$ such that for all $0\leq t_0<t_1\leq T$,
\begin{align*}
\mE\left(\int^{t_1\wedge\tau_R}_{t_0\wedge\tau_R}|b_\sigma(s,Y_s)|^2\dif s\Big|\sF_{t_0\wedge\tau_R}\right)
\leq C \|\det(\sigma\sigma^*)^{-1/(d+1)}|b_\sigma|^2\|_{L^{d+1}([t_0,t_1]\times B_R)}.
\end{align*}
For $|x|\leq R$ and $s\in[0,T]$, noticing that 
$$
|(\sigma\sigma^*)^{-1}(s,x)|\leq\|\sigma\|^{2d-2}_{L^\infty([0,T]\times B_R)}[\det(\sigma\sigma^*)(s,x)]^{-1},
$$
we have
$$
|b_\sigma(s,x)|\leq \|\sigma\|^{2d-1}_{L^\infty([0,T]\times B_R)}\|b\|_{L^\infty([0,T]\times B_R)}[\det(\sigma\sigma^*)(s,x)]^{-1}.
$$
Hence,
$$
\mE\left(\int^{t_1\wedge\tau_R}_{t_0\wedge\tau_R}|b_\sigma(s,Y_s)|^2\dif s\Big|\sF_{t_0\wedge\tau_R}\right)
\leq C \|\det(\sigma\sigma^*)^{-1}\|^{\frac{2d+3}{d+1}}_{L^{2d+3}([t_0,t_1]\times B_R)}.
$$
By Lemma \ref{Le1}, we get \eqref{LL2}. Thus,
by Girsanov's theorem, the process
$$
\tilde W_t:=W_t-\int^{t\wedge\tau_R}_0(\sigma^*(\sigma\sigma^*)^{-1}b)(s,Y_s)\dif s
$$
is still a Brownian motion under $\mQ_R$. In other words, under $\mQ_R$,
$(Y_t, \tilde W_t)$ solves the following SDE:
$$
Y_{t\wedge\tau_R}=y+\int^{t\wedge\tau_R}_0\sigma(s,Y_s)\dif \tilde W_s+\int^{t\wedge\tau_R}_0b(s, Y_s)\dif s.
$$
On the other hand, since $b$ and $\sigma$ are linear growth, it is standard to show that
\begin{align}\label{HJ0}
\mQ_R(\tau_R\leq T)\leq\mE^{\mQ_R}\left(\sup_{t\in[0,T]}|Y_{t\wedge\tau_R}|^2\right)/R^2\leq C/R^2,
\end{align}
where $C$ is independent of $R$.
For $m\in\mN$, let $\mP_m$ be the law of $Y$ in the space of continuous functions under $\mQ_m$
and $\tau_m$ the exit time of canonical process $X_t(\omega)$ from the ball $B_m$. Then by \eqref{HJ0}, for any $T>0$,
$$
\lim_{m\to\infty}\mP_m(\tau_m\leq T)=0.
$$
By Theorem \ref{Ext}, we can extend $\mP_m$ to a probability measure on $\mC$ so that
$\mP=\mP_m$ on $\sB_{\tau_m}(\mC)$.
The proof is complete.
\end{proof}

{\bf Acknowledgement: } The authors would like to thank Professor Feng-Yu Wang for quite useful suggestions about the conditions of Theorem \ref{thm1}.


\begin{thebibliography}{99}

\bibitem{Ch-Ja}N. Champagnat, P.-E. Jabin: Strong solutions to stochastic differential equations with rough coefficients.
To appear in {\it Annals of Probability}.

\bibitem{C} A. Constantin: On the existence and pathwise uniqueness of solutions of stochastic differential equations.
{\it Stochastics and Stochastics Rep.} 56 (1996), no. 3-4, 227-239.

\bibitem{Fa-Zh}S. Fang and T. Zhang: A study of a class of stochastic differential equations with non-Lipschitzian coefficients. 
{\it Probab. Theory Related Fields} 132 (2005), no. 3, 356-390.

\bibitem{K} N. V. Krylov: {\it Controlled diffusion processes}. Translated from the Russian by A. B. Aries. Applications of Mathematics,
14. Springer-Verlag, New York-Berlin, 1980. xii+308 pp.

\bibitem{K-R} N. V. Krylov, M. R\"ockner: Strong solutions of stochastic equations with singular time dependent drift. {\it Probab. Theory Related Fields} 131 (2005), no. 2, 154-196.

\bibitem{Ku}K. S. Kumar: A class of degenerate stochastic differential equations with non-Lipschitz coefficients.
{\it Proceedings - Mathematical Sciences}, 123 (2013), (3) :443-454.

\bibitem{Lu}D. J. Luo: Pathwise uniqueness of multi-dimensional stochastic differential equations with H ̈older diffusion coefficients.
{\it Front. Math. China} 2011, 6(1): 129-136.

\bibitem{Po}N.I. Portenko: {\it Generalized diffusion processes}. Nauka, Moscow, 1982 In Russian;
English translation: Amer. Math. Soc. Provdence, Rhode Island, 1990.

\bibitem{Sc}M. Scheutzow: A stochastic Gronwall's lemma. {\it Infinite Dimensional Analysis, Quantum Probability
and Related Topics}, {\bf 16}, No. 2 (2013) 1350019 (4 pages).

\bibitem{S1} E. M. Stein: {\it Singular integrals and differentiability properties of functions.} Princeton Mathematical Series, No. 30 Princeton University Press, Princeton, N.J. 1970 xiv+290 pp.

\bibitem{S-V} D. W. Stroock, S. S. R. Varadhan: {\it Multidimensional diffusion processes}. Springer-Verlag, Berlin, 1979.

\bibitem{Sw}J.M. Swart: A $2$-dimensional SDE whose solutions are not unique. {\it Electronic Communications in Probability}, 6 (2001)67-71.

\bibitem{V} A. J. Veretennikov: On the strong solutions and explicit formulas for solutions of stochastic differential equations. {\it Mat. Sb.} (N.S.) 111(153) (1980), no. 3, 434-452, 480.

\bibitem{W-Y}  S. Watanabe, T. Yamada: On the uniqueness of solutions of stochastic differential equations, I, II. {\it J. Math. Kyoto Univ.} 11 (1971), 155-167 and 553-563.

\bibitem{X-Zh} L. Xie, X. Zhang: Ergodicity of stochastic differential equations with jumps and singular coefficients. arXiv:1705.07402v1.

\bibitem{Z} A.-K Zvonkin: A transformation of the phase space of a diffusion process that will remove the drift. {\it Mat. Sb.} 93(135) (1974), 129-149, 152.

\bibitem{Zh2} X. Zhang: Strong solutions of SDEs with singular drift and Sobolev diffusion coefficients. {\it Stochastic Process. Appl.} 115 (2005), no. 11, 1805-1818.

\bibitem{Zh3} X. Zhang: Stochastic homeomorphism flows of SDEs with singular drifts and Sobolev diffusion coefficients.
{\it Electron. J. Probab.} 16 (2011), no. 38, 1096-1116.

\bibitem{Zh33}X. Zhang: Well-posedness and large deviation for degenerate SDEs with Sobolev coefficients.
{\it Rev. Mate. Ibv.}, Vol. 29, no.1, pp. 25-52(2013).

\bibitem{Zh4} X. Zhang: Stochastic differential equations with Sobolev diffusion and singular drifts and applications.
{\it Annals of Applied Probab.} Vol. 26, No. 5, 2697-2732(2016).
\end{thebibliography}
\end{document}